\documentclass[11pt]{article}

\usepackage[utf8]{inputenc}
\usepackage[T1]{fontenc}
\usepackage{amsmath,amsthm,amssymb}
\usepackage{mathtools}
\usepackage{mleftright}
\usepackage{graphicx}
\usepackage{booktabs}
\usepackage{multirow}
\usepackage{dsfont} 
\usepackage{hyperref}
\usepackage{url}
\usepackage{csquotes}
\usepackage{tikz}
\usepackage{subcaption}
\usepackage[margin=1in]{geometry}

\usepackage[backend=bibtex,style=numeric,sorting=none]{biblatex}
\theoremstyle{plain}
\newtheorem{theorem}{Theorem}[section]
\newtheorem{lemma}[theorem]{Lemma}
\newtheorem{proposition}[theorem]{Proposition}

\newtheorem{conjecture}[theorem]{Conjecture}
\newtheorem{observation}[theorem]{Observation}

\theoremstyle{definition}
\newtheorem{definition}[theorem]{Definition}
\newtheorem{example}[theorem]{Example}

\theoremstyle{remark}
\newtheorem{remark}[theorem]{Remark}

\newcommand{\ind}{\mathds{1}}

\title{Hierarchical threshold structure in Max-Cut with geometric edge weights}

\author{Nevena Mari\'c\\
\small School of Computing, Union University, Belgrade, Serbia\\
\small \texttt{nmaric@raf.rs}}

\date{}

\newcommand{\paperinfo}{%
\vspace{1em}
\noindent\textbf{Keywords:} Maximum cut, weighted complete graphs, threshold polynomials, hierarchical structure, cut polytope, parametric optimization

\noindent\textbf{2020 Mathematics Subject Classification:} Primary 05C70; Secondary 90C27, 05C22
\vspace{1em}
}

\begin{document}

\maketitle

\begin{abstract}
We study a family of weighted Max-Cut instances on the complete graph $K_n$ in which 
edge weights decrease geometrically in lexicographic order: the $i$-th edge has weight 
$r^{N-i}$ where $N=\binom{n}{2}$. For $r\ge 2$, the lexicographically first cut is optimal; 
for $r=1$, all edges have equal weight and the balanced partition wins. In this paper 
we study the intermediate regime $1< r <2$.

The geometric weighting makes early edges dominant and singles out the 
\emph{$k$-isolated cuts} $C_k=\{1,\dots,k\}\mid\{k+1,\dots,n\}$ as natural candidates 
for optimality. For each $n$ and $k\le\lfloor n/2\rfloor-1$, we define threshold 
polynomials $P^{n,k}(r)$ whose unique roots $r_k(n)\in(1,2)$ determine when $C_k$ and 
$C_{k+1}$ exchange dominance. We prove that, for fixed $n$, these thresholds are strictly decreasing in $k$ and that
$r_k(n)\to 1$ as $n\to\infty$.  As our main result, we show that for $r\in(r_k(n),r_{k-1}(n))$ the cut $C_k$ achieves 
maximum weight among all isolated cuts, yielding a sharp \enquote{phase diagram} for the isolated-cut family.

We conjecture that isolated cuts are globally optimal among all 
$2^{n-1}$ cuts when $n\ge 7$; all counterexamples for small $n$ are characterized 
completely, and extensive computations for $n\le 100$ support the conjecture.
\end{abstract}

\paperinfo

\section{Introduction}

Consider a graph with weighted edges. The maximum cut problem (Max-Cut) is to partition 
the vertices into two sets so as to maximize the total weight of edges between them. While NP-hard in general, structured instances sometimes admit explicit solutions.

In this paper we study one such family on the complete graph $K_n$. 
Order the edges lexicographically
\[
(1,2),(1,3),\ldots,(1,n),(2,3),\ldots,(n-1,n),
\]
and assign the $i$-th edge the weight $r^{N-i}$, where $N=\binom{n}{2}$ and $r>1$.
Thus weights decrease geometrically along the lexicographic order, with $(1,2)$ receiving
the maximum weight $r^{N-1}$. Related lexicographic cost functions appear, for example,
in the study of local search complexity~\cite{scheder2025pls}.

The geometric weighting favors edges involving small-index vertices, which suggests 
the \emph{$k$-isolated cuts} $C_k$ - partitions $\{1,\ldots,k\}\mid\{k+1,\ldots,n\}$ - as 
natural candidates for optimality.

The extreme regimes are straightforward. When $r\ge 2$, the weight of $(1,2)$ exceeds the sum
of all remaining weights, hence $C_1$ is optimal. When $r=1$, all edges have equal weight, 
and the balanced cut $C_{\lfloor n/2 \rfloor}$ wins. The main question is what happens 
in the intermediate regime $1<r<2$.

Our first contribution is an explicit threshold description within this family. For each
$n$ and each $k\le \lfloor n/2\rfloor-1$, we define a threshold polynomial $P^{n,k}(r)$
whose unique root $r_k(n)\in(1,2)$ marks the parameter value at which $C_k$ and $C_{k+1}$
exchange dominance. The thresholds satisfy
\[
r_1(n) > r_2(n) > \cdots > r_{\lfloor n/2\rfloor-1}(n) > 1,
\]
and $r_k(n)\to 1$ as $n\to\infty$. As our main result, we prove that for $r$ between
consecutive thresholds, the corresponding cut $C_k$ maximizes the weight among all isolated cuts,
yielding an explicit description of the best isolated cut as a function of $r$.

Whether an isolated cut is globally optimal among all $2^{n-1}$ cuts remains open. We conjecture
that this holds for $n\ge 7$; exhaustive enumeration for $n \le 21$ and targeted verification 
for $n \le 100$ support this conjecture. The bound $n \ge 7$ is sharp: for $n\in\{4,5,6\}$ 
we identify counterexamples, given by near-isolated cuts of the form $\{1,\ldots,k,n\}$.

We also compare the exact optimum for this family with standard general Max-Cut bounds and show
that such bounds remain far from tight in this structured setting; see Section~\ref{sec:conclusion}.

\medskip
\noindent\textbf{The rest of the paper is organized as follows.}
Section~\ref{sec:prelim} sets notation and derives closed forms for the cut weights used throughout.
Section~\ref{sec:observations} summarizes initial observations and provides motivation for the main result.
Section~\ref{sec:thresholds} defines the threshold polynomials and proves their uniqueness and recursion, leading to the resulting threshold structure among isolated cuts; it also contains our main result on optimality within the isolated family.
Section~\ref{sec:global} discusses the global-optimality conjecture, including a complete classification of the counterexamples for $n\in\{4,5,6\}$ and supporting computations.
Section~\ref{sec:conclusion} compares our findings with general Max-Cut bounds and outlines further questions.
Heuristic derivations and additional computational details are provided in Appendices~A and~B, respectively.

\section{Preliminaries}\label{sec:prelim}

Let $K_n$ be the complete graph on $n$ vertices $[n]=\{1,\ldots,n\}$ with $N = \binom{n}{2}$ edges. 
A \textbf{cut vector} $\delta(S)$ of $K_n$ is defined for every $S \subseteq [n]$ by
$\delta(S)_{ij} = 1$ if $|S \cap \{i,j\}| = 1$ and $0$ otherwise.  
Each cut corresponds to a partition $S \cup S^c$ of $[n]$, with $2^{n-1}$ distinct cuts up to 
complementation. We order edges lexicographically: $(1,2), (1,3), \ldots, (1,n), (2,3), \ldots, (n-1,n)$.
The edge $(i,j)$ with $i < j$ has lexicographic index
\begin{equation}\label{eq:edge_index}
\mathrm{idx}(i,j) = (i-1)n - \binom{i}{2} + (j-i).
\end{equation}

Each cut corresponds to a $0$-$1$ $n$-tuple where $1$ at position $i$ indicates that vertex $i$ 
belongs to the cut. Among the two complementary representations, we choose the one that includes 
vertex~1. For an $n$-tuple ${\bf{x}} = (x_1,\ldots, x_n)$ with $x_1 = 1$, the cut vector is
\begin{align} \label{def:indikatori}
\delta({\bf{x}}) = (\ind(x_1 \neq x_2), \ind(x_1 \neq x_3), \ldots, \ind(x_{n-1} \neq x_n)) \in \{0,1\}^N,
\end{align}
where $\ind(A) = 1$ if $A$ is true and $0$ otherwise. This extends the notation $\delta(S)$ to binary strings, where $x_i = \ind(i \in S)$. 
If the $\bf{x}$'s are ordered 
lexicographically, then the corresponding $\delta(\bf{x})$'s appear in reverse lexicographic 
order. This follows from \cite{maric2025explicit}, 
where lexicographic order of $\mathbf{1} - \delta(\bf{x})$ is established. Table~\ref{table:n4} shows all cut vectors for $n=4$.

Note that we identify a 0-1 vector with the corresponding 0-1 string when no ambiguity arises. We write $1^i 0^j$ for a string of $i$ ones followed by $j$ zeros; for example, $1^k 0^{n-k}$ denotes $k$ ones followed by $n-k$ zeros.

\begin{table}[!ht]
\centering
\begin{tabular}{|c|c|c|}
\hline
order & $\bf{x}$ & cut vector \,$ \delta(\bf{x})$ \\
\hline
1 & 1000 & 111000  \\
2& 1001 & 110011  \\
3 & 1010 & 101101   \\
4& 1011 & 100110  \\
5 & 1100 & 011110  \\
6 & 1101 & 010101  \\
7& 1110 & 001011  \\
8& 1111 & 000000  \\
\hline
\end{tabular}
\caption{Cut vectors for $n=4$. The binary representation $\bf{x}$ and corresponding cut vector $\delta(\bf{x})$ are shown for all $2^{n-1}=8$ cuts.}
\label{table:n4}
\end{table}

\begin{definition}
The \emph{$k$-isolated cut} $C_k$ is the cut vector corresponding 
to the partition $\{1,\ldots,k\} \mid \{k+1,\ldots,n\}$. In terms of vertex indicators, $C_k = \delta(1^{k} 0^{n-k})$.
\end{definition}

\begin{remark}
Among all $2^{n-1}$ cuts ordered lexicographically by vertex membership, 
$C_k$ appears at position $1 + \sum_{j=0}^{k-2} 2^{n-2-j}$. In particular, 
$C_1$ is first, $C_2$ is at position $1 + 2^{n-2}$, and $C_3$ at position $1 + 2^{n-2} + 2^{n-3}$.
\end{remark}

\begin{definition}[Weight function]
For a parameter $r > 1$, define the geometric weight $w_i = r^{N-i}$ for the $i$-th edge 
in lexicographic order. For a cut vector $\delta \in \{0,1\}^N$, define
\[
W^n(\delta;r) = \sum_{i=1}^N \delta_i \cdot r^{N-i}.
\]
\end{definition}

We focus on the regime $1 < r < 2$.

\begin{definition} \label{def:threshold}
The \emph{threshold} $r_k(n)$ is the value of $r$ at which the $k$-isolated and $(k+1)$-isolated 
cuts have equal weight: $W^n (C_k;r_k(n)) = W^n (C_{k+1};r_k(n))$.
\end{definition}

\section{Computational Observations}\label{sec:observations}

Before developing the theory, we present computational observations that motivate the main results.
For $n \leq 21$, we evaluated cut weights over a grid of $r\in(1,2)$ (mesh size $0.001$) and tracked the maximizer. This grid-based verification supports the observations numerically but does not constitute a proof over the continuum; see Section~\ref{sec:computational} for details.

\begin{observation}\label{obs:isolated_transitions}
Fix $n\ge 7$. As $r$ increases from $1^+$ to $2^-$, the maximizer \emph{within the family of $k$-isolated cuts}
transitions successively through $C_{\lfloor n/2\rfloor}, C_{\lfloor n/2\rfloor-1},\ldots, C_1$.
\end{observation}

Table~\ref{tab:transitions} illustrates Observation~\ref{obs:isolated_transitions} for $n=7,\dots,12$.
The rightmost column lists the sequence of values of $k$ for which the $k$-isolated cut $C_k$
maximizes the weight (over the relevant range of $r$). 
The arrows indicate the order in which the maximizing
isolated cut changes as $r$ increases from $1^+$ to $2^-$.

\begin{table}[!ht]
\centering
\caption{Observed transition sequences of the maximizing isolated cut $C_k$ as $r$ increases in $(1,2)$, for $n=7$ to $n=12$.}
\label{tab:transitions}
\begin{tabular}{c|c}
\toprule
$n$ & $k$ values \\
\midrule
7  &  3 $\to$ 2 $\to$ 1 \\
8  &  4 $\to$ 3 $\to$ 2 $\to$ 1 \\
9  &  4 $\to$ 3 $\to$ 2 $\to$ 1 \\
10 &  5 $\to$ 4 $\to$ 3 $\to$ 2 $\to$ 1 \\
11 &  5 $\to$ 4 $\to$ 3 $\to$ 2 $\to$ 1 \\
12 &  6 $\to$ 5 $\to$ 4 $\to$ 3 $\to$ 2 $\to$ 1 \\
\bottomrule
\end{tabular}
\end{table}

The transitions occur at threshold values $r_k(n)$ from Definition~\ref{def:threshold}; theorems in Section~\ref{sec:thresholds} show that,
within the isolated family, these thresholds are unique and strictly ordered.

Section~\ref{sec:thresholds} provides a rigorous threshold description of this behavior and proves
that, between consecutive thresholds, the corresponding cut $C_k$ is optimal among all isolated cuts.

For $n \leq 6$, non-isolated cuts can be optimal in intervals near $r=1$; these cases are treated in
Section~\ref{sec:global}. A stronger global-optimality statement, suggested by exhaustive enumeration for $n\le 21$
and additional verification for $n\le 100$, is formulated as a conjecture in Section~\ref{sec:global}.


\section{Polynomial Characterization}\label{sec:thresholds}

The computational observations suggest that threshold values play a fundamental role. We now develop the mathematical machinery to characterize these thresholds precisely.

\begin{definition}
For $k \geq 1$, define the \emph{threshold polynomial}
\[
P^{n,k}(r) = W^n(C_k;r) - W^n(C_{k+1};r).
\]
\end{definition}

\begin{table}[!ht]
\centering
\caption{Comparison of cut vectors for $k=2$ and $k=3$ isolated cuts ($n=7$)}
\begin{tabular}{r|c|l}
\hline
k & $\mathbf{x}$ & Cut vector $C_k=\delta(\mathbf{x})$ \\
\hline
$2$ & \texttt{1100000} & 
\texttt{0\,1\,1\,1\,1\,1\,\,1\,1\,1\,1\,1\,\,0\,0\,0\,0\,\,0\,0\,0\,\,0\,0\,\,0} \\
\hline
$3$ & \texttt{1110000} & 
\texttt{0\,0\,1\,1\,1\,1\,\,0\,1\,1\,1\,1\,\,1\,1\,1\,1\,\,0\,0\,0\,\,0\,0\,\,0} \\
& & \texttt{~\,$\uparrow$\,~\,~\,~\,~\,\,$\uparrow$\,~\,~\,~\,\,\,\,\,\,$\uparrow$\,$\uparrow$\,\,$\uparrow$\,$\uparrow$\,\,} \\
\hline
\end{tabular}

\vspace{0.5em}
\footnotesize
Arrows mark positions where cut vectors differ (contributing to polynomial equation).
\end{table}

The roots of $P^{n,k}(r)$ in $(1,2)$ correspond to thresholds $r_k(n)$. We begin by analyzing the simplest case.

\subsection{Base Case: $k=1$}

\begin{proposition}[Explicit Formula for $k=1$]
\label{prop:base_case}
For $n \geq 6$,
\[
P^{n,1}(r) = r^{N-1}-(r^{N-n}+\ldots+r^{N-2n+3}).
\]
\end{proposition}

\begin{proof}
\begin{eqnarray*}
W^n(C_1;r)&=& W^n(1^{(n-1)}0^{(N-n+1)};r) = r^{N-1}+\ldots + r^{N-n+1}\\
W^n(C_2;r)&= & W^n(0\,1^{(n-2)} 1^{(n-2)} 0^{(N-2n+3)};r) \\
&=& r^{N-2}+\ldots + r^{N-n+1}+r^{N-n}+\ldots+ r^{N-2n+3}
\end{eqnarray*}
Subtraction of the two polynomials finishes the proof.
\end{proof}

\begin{lemma}[Exponent Separation]\label{lem:separation}
In $P^{n,k}(r)$, all positive-coefficient terms have larger exponents than 
all negative-coefficient terms. Specifically, if $a_k$ is the smallest 
positive exponent and $b_1$ the largest negative exponent, then 
$a_k - b_1 = n - k \geq 1$.
\end{lemma}
\begin{proof}
The positive terms in $P^{n,k}$ correspond to edges $(i, k+1)$ for $i = 1, \ldots, k$, 
and the negative terms to edges $(k+1, j)$ for $j = k+2, \ldots, n$.

The smallest positive exponent is for edge $(k, k+1)$, and the largest negative 
exponent is for edge $(k+1, k+2)$. By~\eqref{eq:edge_index}, edge $(i,j)$ has exponent $N - \mathrm{idx}(i,j)$. 
Substituting $(i,j) = (k, k+1)$ and $(i,j) = (k+1, k+2)$:
\begin{align*}
a_k &= N - 1 - \frac{(k-1)(2n-k)}{2}, \\
b_1 &= N - 1 - \frac{k(2n-k-1)}{2}.
\end{align*}
Thus
\[
a_k - b_1 = \frac{k(2n-k-1) - (k-1)(2n-k)}{2} = \frac{2n - 2k}{2} = n - k \geq 1. \qedhere
\]
\end{proof}

This separation property is crucial for applying Descartes' Rule of Signs to establish uniqueness of roots.

\begin{theorem}[Properties of $r_1(n)$] 
\label{thm:r1_properties}
\begin{enumerate}
\item[(a)] $P^{n,1}(r)$  has a unique root $r_1(n) \in (1,2)$.
\item[(b)] $r_1(n)$ is strictly decreasing in $n$.
\item[(c)] $\lim_{n \to \infty} r_1(n) = 1$.
\end{enumerate}
\end{theorem}

\begin{proof}
\begin{enumerate}
\item[(a)]
\textbf{Existence.} 
We first show that a root exists by evaluating $P^{n,1}$ (using Proposition \ref{prop:base_case}) at the endpoints.

At $r=1$:
\[
P^{n,1}(1) = 1 - \sum_{i=0}^{n-3} 1 = 1 - (n-2) = 3 - n < 0, \quad \text{for } n \geq 4.
\]

At $r=2$:
\[
P^{n,1}(2) = 2^{N-1} - \sum_{i=0}^{n-3} 2^{N-n-i} = 2^{N-2} + 2^{N-n} > 0.
\]
By the Intermediate Value Theorem (IVT), $P^{n,1}$ has at least one root in $(1,2)$.

\textbf{Uniqueness.} 
We apply Descartes' Rule of Signs (see for example \cite{anderson1998descartes}). Writing $P^{n,1}(r)$ in descending order of powers:
\[
P^{n,1}(r) = r^{N-1} + 0 \cdot r^{N-2} + \cdots + 0 \cdot r^{N-n+1} - r^{N-n} - r^{N-n-1} - \cdots - r^{N-2n+3}.
\]
The sequence of nonzero coefficients is $(+1, -1, -1, \ldots, -1)$; by Lemma~\ref{lem:separation} all positive-coefficient terms have larger exponents than all negative-coefficient terms, so there is exactly one sign change. By Descartes' rule, $P^{n,1}$ has at most one positive root; for a recent application of this technique to graph polynomials, see Brown and George~\cite{brown2025roots}.

 Combined with the existence proved above, $P^{n,1}$ has exactly one positive root, and it lies in $(1,2)$. Figure~\ref{fig:polynomials} illustrates this for $n = 8$.

\item[(b)] \textbf{$r_1(n)$ is strictly decreasing in $n$.}

For $n=5,6,7$ the monotonicity is verified by direct computation: $r_1(5)=1.249852, r_1(6)= 1.243347 ,r_1(7)= 1.229318$.

We prove that $r_1(n+1) < r_1(n)$ for all $n \geq 7$.
Note that 
 \[      
 P^{n,1}(x)= \frac{x^{N-2n+3}}{x-1} f_n(x).
 \]
 where $f_n(x) = x^{2n-3} - x^{2n-4} - x^{n-2} + 1$.

Clearly $P^{n,1}(x) = 0$ if and only if $f_n(x) = 0$ for $x \in (1,2)$ and, by part (a), $ r_1(n)$ denotes the unique root of $f_n$ in $(1,2)$.

\textbf{Claim:} $r_1(n)^{n-1} > 3$ for $n \geq 7$.

Let $q_n = 3^{1/(n-1)}$. Then:
\begin{align*}
f_n(q_n) &= q_n^{2n-3} - q_n^{2n-4} - q_n^{n-2} + 1 \\
&= 3^{2-\frac{1}{n-1}} - 3^{2-\frac{2}{n-1}} - 3^{1-\frac{1}{n-1}} + 1 = \frac{9}{q_n} - \frac{9}{q_n^2} - \frac{3}{q_n} + 1 \\
&= -\frac{9}{q_n^2} +\frac{6}{q_n} +1.
\end{align*}

We need $f_n(q_n) < 0$, which is equivalent to
$$ q_n^2  + 6q_n - 9 < 0 \quad \Leftrightarrow \quad q_n < -3 + 3\sqrt{2} \approx 1.2426.$$

Since $3^{1/6} \approx 1.2009 < 1.2426$, this holds for $n \geq 7$. (Note: $3^{1/5} \approx 1.2457$.)

Since $f_n(q_n) < 0$ and $f_n(r_1(n)) = 0$, we have $r_1(n) > q_n$, thus $r_1(n)^{n-1} > 3$.

{\bf Now we evaluate $f_{n+1}(r_1(n))$.} Since $f_n(r_1(n)) = 0$, we have $r_1(n)^{2n-3} - r_1(n)^{2n-4} = r_1(n)^{n-2} -1$. Multiplying by $r_1(n)^2$:
$$r_1(n)^{2n-1} - r_1(n)^{2n-2} = r_1(n)^n -r_1(n)^2.$$

Therefore:
\begin{align*}
f_{n+1}(r_1(n)) &= r_1(n)^{2n-1} - r_1(n)^{2n-2} - r_1(n)^{n-1} + 1 \\
&= (r_1(n)^n -r_1(n)^2) - r_1(n)^{n-1} + 1 \\
&= (r_1(n) - 1)( r_1(n)^{n-1}-r_1(n)-1).
\end{align*}

For $n \geq 7$, we have $r_1(n)^{n-1} > 3 > r_1(n) + 1$ (as $r_1(n) < 2$). Thus:
$$f_{n+1}(r_1(n)) = (r_1(n) - 1) \underbrace{(r_1(n)^{n-1}-r_1(n)-1)}_{> 0} > 0.$$

Since $f_{n+1}$ is continuous with $f_{n+1}(r_1(n)) >  0$ and $f_{n+1}(2) > 0$, and since $f_{n+1}(1)=0$ and $f_{n+1}'(1)=2-n<0$ for $n \geq 7$, so that $f_{n+1}<0$ just above $1$, the Intermediate Value Theorem gives  $r_1(n+1) \in (1, r_1(n))$.

\item[(c)] \textbf{$\lim_{n \to \infty} r_1(n) = 1$.}

At the root, $r_1(n)^{N-1} = \sum_{i=0}^{n-3} r_1(n)^{N-n-i} < (n-2) \cdot r_1(n)^{N-n}$, so $r_1(n)^{n-1} < n-2$, giving $1 < r_1(n) < (n-2)^{1/(n-1)} \to 1$. The result follows by the squeeze theorem.

\end{enumerate}
\end{proof}

The decay of $r_1(n)$ toward 1 is visible in Figure~\ref{fig:rk_vs_n}.

Having established the properties for $k=1$, we now develop a recursive structure that extends these results to all $k$.

\subsection{Recursive Structure for $k \geq 2$}

When comparing cut vectors across different complete graphs, we write $\delta(\mathbf{x})_{K_m}$ 
to indicate that the cut vector is computed in $K_m$. We use $\|$ to denote concatenation of binary strings.

\begin{lemma}[Cut Vector Concatenation]
\label{lem:concatenation}
The cut vector for the $k$-isolated cut in $K_n$ satisfies
\[
\delta\mleft(\underbrace{11\ldots1}_k\underbrace{00\ldots0}_{n-k}\mright)_{K_n} = \underbrace{00\ldots0}_{k-1}\underbrace{11\ldots1}_{n-k} \,\|\, \delta\mleft(\underbrace{11\ldots1}_{k-1}\underbrace{00\ldots0}_{n-k}\mright)_{K_{n-1}}.
\]
\end{lemma}

\begin{proof}
Let $\mathbf{x} = (1,\ldots,1,0,\ldots,0)$ with $k$ ones followed by $n-k$ zeros. The cut vector $\delta(\mathbf{x})$ has $\binom{n}{2}$ entries corresponding to edges in lexicographic order.

The first $n-1$ entries correspond to edges $(1,2), (1,3), \ldots, (1,n)$. For edge $(1,j)$:
\[
\delta(\mathbf{x})_{1j} = \ind(x_1 \neq x_j) = \ind(1 \neq x_j) = 
\begin{cases}
0 & \text{if } j \leq k, \\
1 & \text{if } j > k.
\end{cases}
\]
Thus the first $n-1$ entries are $\underbrace{00\ldots0}_{k-1}\underbrace{11\ldots1}_{n-k}$.

The remaining $\binom{n-1}{2}$ entries correspond to edges $(i,j)$ with $2 \leq i < j \leq n$. These are precisely the edges of $K_{n-1}$ on vertices $\{2, \ldots, n\}$. For such edges:
\[
\delta(\mathbf{x})_{ij} = \ind(x_i \neq x_j),
\]
which depends only on $(x_2, \ldots, x_n) = (\underbrace{1,\ldots,1}_{k-1}, \underbrace{0,\ldots,0}_{n-k})$. This is exactly $\delta(\underbrace{11\ldots1}_{k-1}\underbrace{00\ldots0}_{n-k})_{K_{n-1}}$.
\end{proof}

\begin{theorem}[Recursive Formula]
\label{thm:recursion}
For all $k \geq 2$ and $n \geq k+1$,
\[
P^{n,k}(r) = r^{N-k} + P^{n-1,k-1}(r),
\]
where $N = \binom{n}{2}$.
\end{theorem}

\begin{proof}
Using Lemma \ref{lem:concatenation} for $k$ and $k+1$:
\begin{eqnarray*}
W^n(C_k; r) &=& W^n \!\mleft( \underbrace{00\ldots0}_{k-1}\underbrace{11\ldots1}_{n-k} \,\|\, \delta\mleft(\underbrace{11\ldots1}_{k-1}\underbrace{00\ldots0}_{n-k}\mright)_{K_{n-1}};\, r \mright), \\
W^n(C_{k+1}; r) &=& W^n \!\mleft( \underbrace{00\ldots00}_{k}\underbrace{1\ldots1}_{n-k-1} \,\|\, \delta\mleft(\underbrace{11\ldots1}_{k}\underbrace{00\ldots0}_{n-k-1}\mright)_{K_{n-1}};\, r \mright).
\end{eqnarray*}
Subtracting the lower from the upper expression, we obtain
\begin{equation*}
 r^{N-k}+ W^{n-1}(C_{k-1};r)-W^{n-1}(C_k;r),
\end{equation*}
and the recursive relationship follows directly.
\end{proof}

\subsection{Properties of the Hierarchy}\label{sec:monotonicity}

With the recursive formula in hand, we can now establish the key properties of the threshold sequence: uniqueness, monotonicity, and convergence.

\begin{theorem}[Unique Roots]
\label{thm:unique_roots}
For each $n \geq 6$ and $1 \leq k \leq \lfloor n/2 \rfloor - 1$, the polynomial $P^{n,k}(r)$ has a unique root $r_k(n)$ in the interval $(1,2)$.
\end{theorem}

\begin{proof}
We prove this for all $k$ simultaneously using the recursive structure.

\textbf{Step 1: Explicit expansion.}
By repeatedly applying the recursion $P^{n,k}(r) = r^{N-k} + P^{n-1,k-1}(r)$, we obtain
\[
P^{n,k}(r) = r^{\binom{n}{2}-k} + r^{\binom{n-1}{2}-(k-1)} + \cdots + r^{\binom{m+1}{2}-2} + P^{m,1}(r),
\]
where $m = n - k + 1$. Substituting the base case formula for $P^{m,1}(r)$:
\[
P^{n,k}(r) = \underbrace{r^{\binom{n}{2}-k} + \cdots + r^{\binom{m+1}{2}-2} + r^{\binom{m}{2}-1}}_{\text{positive coefficients}} - \underbrace{\sum_{i=0}^{m-3} r^{\binom{m}{2}-m-i}}_{\text{negative coefficients}}.
\]

\textbf{Step 2: Existence of a root.}
At $r = 1$, all terms equal $1$, so
\[
P^{n,k}(1) = k  - (m-2) = k - (n-k+1) + 2 = 2k - n + 1.
\]
For $k \leq \lfloor n/2 \rfloor - 1$, we have $2k \leq n - 2$, thus $P^{n,k}(1) \leq -1 < 0$.

At $r = 2$, by Lemma~\ref{lem:separation} the smallest positive exponent satisfies $a_k = b_1 + (n-k)$, so the smallest positive term alone is $2^{a_k} = 2^{n-k} \cdot 2^{b_1}$, while the entire negative sum is at most $(n-k-1) \cdot 2^{b_1} < 2^{n-k} \cdot 2^{b_1}$. Hence $P^{n,k}(2) > 0$.

By the Intermediate Value Theorem, $P^{n,k}$ has at least one root in $(1,2)$.

\textbf{Step 3: Uniqueness.}
From Step 1, all nonzero coefficients of $P^{n,k}$ are $\pm 1$. By Lemma~\ref{lem:separation}, $a_k > b_1$, so in decreasing order of exponents the coefficients have exactly one sign change. By Descartes' Rule of Signs, $P^{n,k}(r)$ has at most one positive root. Combined with existence, $P^{n,k}$ has exactly one positive root, which lies in $(1,2)$.
\end{proof}

\begin{figure}[!ht]
\centering
\includegraphics[width=0.8\textwidth]{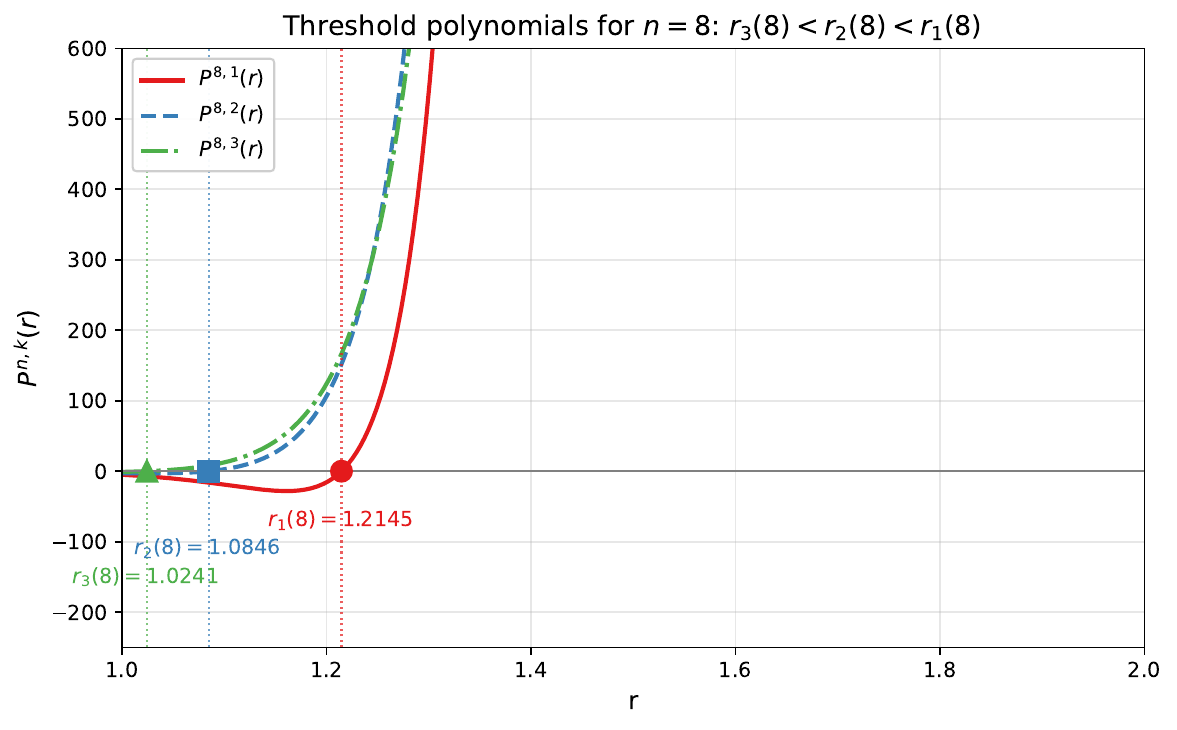}
\caption{Threshold polynomials $P^{8,1}(r)$, $P^{8,2}(r)$, and $P^{8,3}(r)$ on the interval $(1,2)$. Each polynomial has a unique root in $(1,2)$, with $r_3(8) < r_2(8) < r_1(8)$, illustrating the monotonicity established in Theorem~\ref{thm:monotonicity}.}
\label{fig:polynomials}
\end{figure}

\begin{figure}[!ht]
\centering
\includegraphics[width=0.85\textwidth]{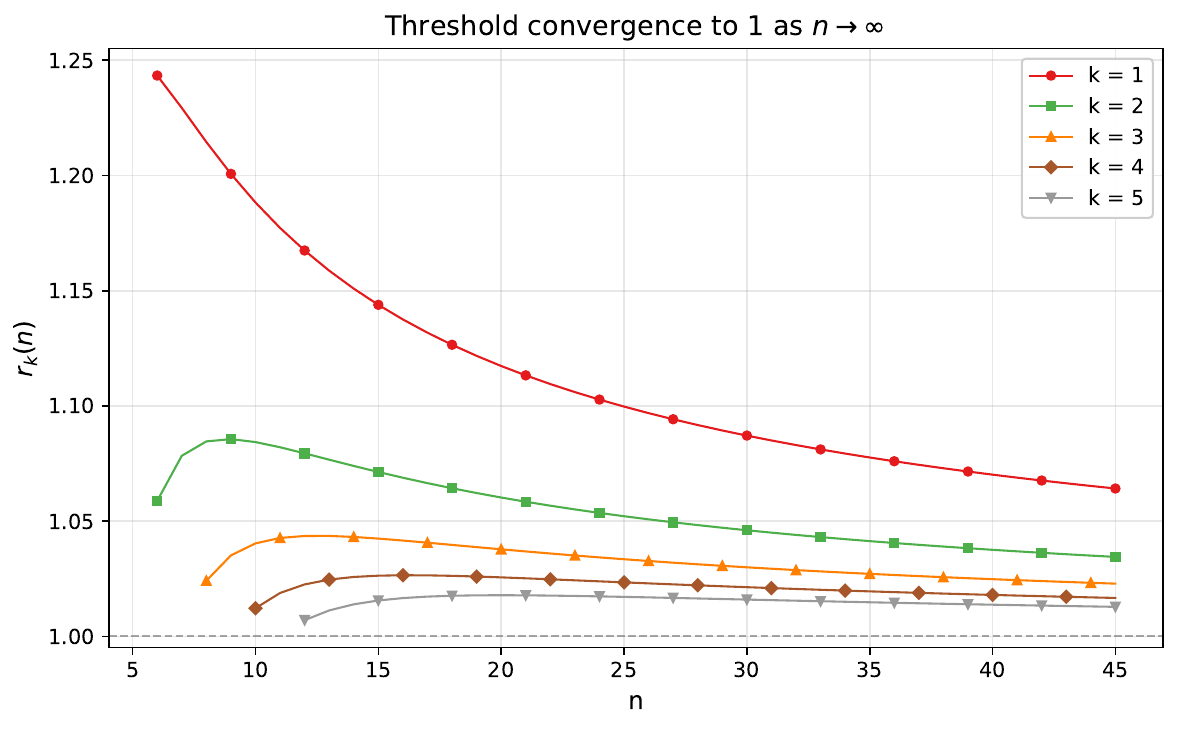}
\caption{Threshold values $r_k(n)$ as a function of $n$ for fixed $k \in \{1,2,3,4,5\}$. All thresholds converge to $1$ as $n \to \infty$. Note that for $k \geq 2$, the sequences are not monotone in $n$.}
\label{fig:rk_vs_n}
\end{figure}

\begin{proposition}
\label{cor:reduction_base}
For $n \geq 6$, $1 \leq k \leq \lfloor n/2 \rfloor - 1$, and $r > 1$,
\[
r \cdot P^{n,k+1}(r) - P^{n,k}(r) > 0.
\]
\end{proposition}

\begin{proof}
From the recursion $P^{n,k}(r) = r^{N-k} + P^{n-1,k-1}(r)$ (Theorem~\ref{thm:recursion}):
\begin{align*}
r \cdot P^{n,k+1}(r) - P^{n,k}(r) &= r\mleft(r^{N-k-1} + P^{n-1,k}(r)\mright) - \mleft(r^{N-k} + P^{n-1,k-1}(r)\mright) \\
&= r \cdot P^{n-1,k}(r) - P^{n-1,k-1}(r).
\end{align*}
Applying this reduction $k-1$ times yields
\[
r \cdot P^{n,k+1}(r) - P^{n,k}(r) = r \cdot P^{m,2}(r) - P^{m,1}(r),
\]
where $m = n-k+1 \geq 5$. Direct calculation gives
\[
r \cdot P^{m,2}(r) - P^{m,1}(r) = \underbrace{(r^{M-m+1} + \cdots + r^{M-2m+3})}_{m-1 \text{ terms}} - \underbrace{(r^{M-2m+3} + \cdots + r^{M-3m+7})}_{m-3 \text{ terms}},
\]
where $M = \binom{m}{2}$. The smallest positive exponent equals the largest negative exponent; 
after cancellation, $m-2$ positive terms with larger exponents dominate the $m-4$ remaining 
negative terms. Hence the expression is positive for $r > 1$.
\end{proof}

The inequality $r \cdot P^{n,k+1}(r) > P^{n,k}(r)$ implies that when $P^{n,k}$ vanishes, 
$P^{n,k+1}$ must be positive which is the key observation for establishing monotonicity.

\begin{theorem}[Monotonicity and Convergence]
\label{thm:monotonicity}\label{thm:convergence}
For $n \geq 6$:
\begin{enumerate}
\item[(a)] The sequence of thresholds is strictly decreasing: $r_{k+1}(n) < r_k(n)$ for $1 \leq k \leq \lfloor n/2 \rfloor - 2$.
\item[(b)] For all $k \geq 1$, $\displaystyle\lim_{n \to \infty} r_k(n) = 1$.
\end{enumerate}
\end{theorem}

\begin{proof}
(a) At $r = r_k(n)$, we have $P^{n,k}(r_k(n)) = 0$ by definition. Therefore:
\[
r_k(n) \cdot P^{n,k+1}(r_k(n)) = r_k(n) \cdot P^{n,k+1}(r_k(n)) - P^{n,k}(r_k(n)) > 0
\]
by Proposition \ref{cor:reduction_base}. Since $r_k(n) > 1$, we conclude $P^{n,k+1}(r_k(n)) > 0$.

By Theorem \ref{thm:unique_roots}, $P^{n,k+1}(r)$ has a unique root $r_{k+1}(n)$ in $(1,2)$. Since $P^{n,k+1}(1) \leq 0$ (from the proof of Theorem \ref{thm:unique_roots}) and $P^{n,k+1}(r_k(n)) > 0$, the root must satisfy $r_{k+1}(n) < r_k(n)$.

(b) By part (a), $1 < r_k(n) < r_1(n)$ for all $k \geq 2$. Since $r_1(n) \to 1$ by Theorem \ref{thm:r1_properties}(c), the squeeze theorem gives $r_k(n) \to 1$.
\end{proof}

\begin{figure}[!ht]
\centering
\includegraphics[width=0.8\textwidth]{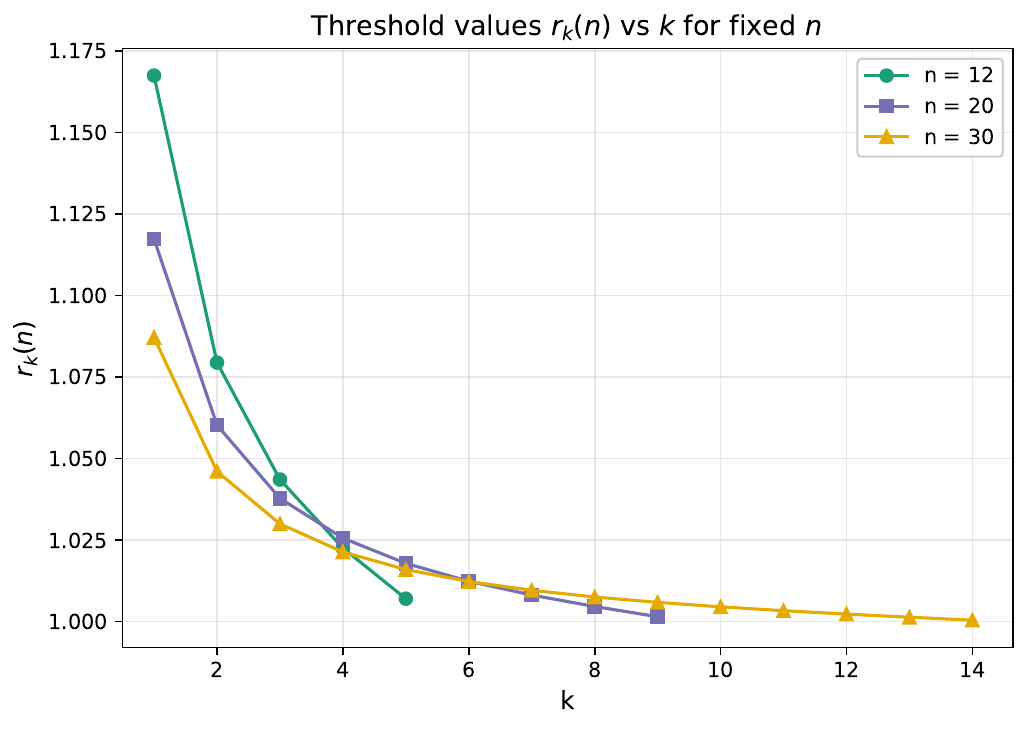}
\caption{Threshold values $r_k(n)$ as a function of $k$ for fixed $n \in \{12, 20, 30\}$. For each $n$, the sequence $r_1(n) > r_2(n) > \cdots > r_{\lfloor n/2 \rfloor - 1}(n)$ is strictly decreasing, as established in Theorem~\ref{thm:monotonicity}.}
\label{fig:rk_vs_k}
\end{figure}

Together, Theorems \ref{thm:unique_roots} and \ref{thm:monotonicity} show that the thresholds partition $(1,2)$ into intervals where successive isolated cuts are pairwise optimal. The next result promotes this to optimality over \emph{all} isolated cuts simultaneously.

To state the result cleanly, we adopt the convention $r_0(n) = 2$ (so that the interval $(r_1(n), r_0(n))$ corresponds to the regime where $C_1$ is optimal).

\begin{theorem}[Main result: Optimality among isolated cuts]
\label{thm:isolated_optimality}
For $n \geq 6$, $1 \leq k \leq \lfloor n/2 \rfloor$, and $r \in (r_k(n), r_{k-1}(n))$, the $k$-isolated cut achieves the maximum weight among all $j$-isolated cuts, $j = 1, \ldots, \lfloor n/2 \rfloor$.
\end{theorem}

\begin{proof}
Recall that $P^{n,j}(r) = W^n(C_j;r) - W^n(C_{j+1};r)$. By Theorem \ref{thm:unique_roots}, $P^{n,j}(r)$ has a unique root $r_j(n)$ in $(1,2)$, with $P^{n,j}(1) \leq 0$ and $P^{n,j}(2) > 0$. Therefore:
\begin{itemize}
    \item For $r < r_j(n)$: $P^{n,j}(r) < 0$, hence $W^n(C_j;r) < W^n(C_{j+1};r)$
    \item For $r > r_j(n)$: $P^{n,j}(r) > 0$, hence $W^n(C_j;r) > W^n(C_{j+1};r)$
\end{itemize}

Let $r \in (r_k(n), r_{k-1}(n))$. We show $W^n(k\text{-isolated};r) > W^n(j\text{-isolated};r)$ for all $j \neq k$.

\textbf{Case 1: $j < k$.} By Theorem \ref{thm:monotonicity}, $r_j(n) > r_{j+1}(n) > \cdots > r_{k-1}(n) > r$. Thus $r < r_j(n)$ for all $j < k$, which implies:
\[
W^n(C_j;r) < W^n(C_{j+1};r) < \cdots < W^n(C_{k-1};r) < W^n(C_k;r)
\]

\textbf{Case 2: $j > k$.} By Theorem \ref{thm:monotonicity}, $r > r_k(n) > r_{k+1}(n) > \cdots > r_{j-1}(n)$. Thus $r > r_i(n)$ for all $i \geq k$, which implies:
\[
W^n(C_k;r) > W^n(C_{k+1};r) > \cdots > W^n(C_{j-1};r) > W^n(C_j;r)
\]

Combining both cases, $W^n(C_k;r) > W^n(C_j;r)$ for all $j \neq k$.
\end{proof}

\begin{remark}
At the boundary $r = r_k(n)$, cuts $C_k$ and $C_{k+1}$ have equal weight. 
Note that $C_1$ remains optimal for all $r \geq r_1(n)$, including the trivial 
regime $r \geq 2$ where edge $(1,2)$ alone dominates.
\end{remark}

This theorem provides a complete solution to the optimization problem restricted to 
isolated cuts: for each $r \in (1,2)$, the optimal isolated cut is explicitly identified 
by locating $r$ within the threshold intervals.

\begin{figure}[!ht]
\centering
\subcaptionbox
  {Phase diagram for $n = 8$. Vertical dashed lines indicate thresholds $r_3(8) \approx 1.0241$, $r_2(8) \approx 1.0846$, and $r_1(8) \approx 1.2145$.\label{fig:sub1}}
  [.48\linewidth]
  {\includegraphics[width=.9\linewidth]{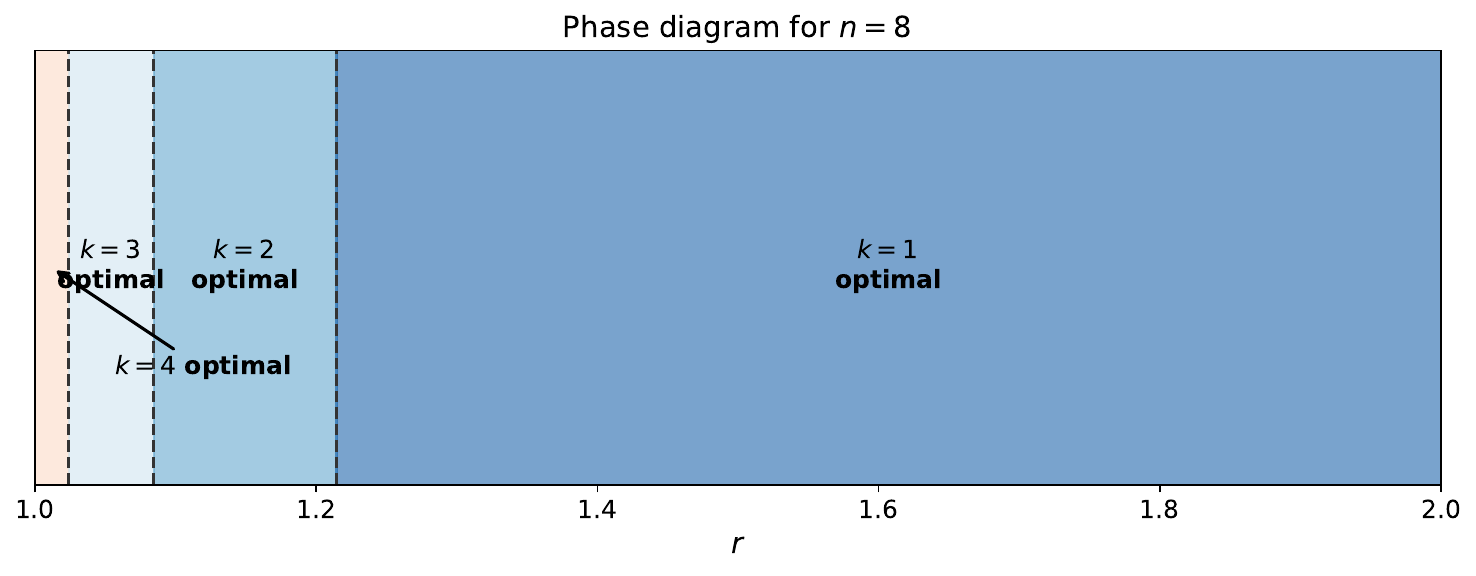}}%
\hfill
\subcaptionbox
  {Phase diagram for $n = 20$\label{fig:sub2}}
  [.48\linewidth]
  {\includegraphics[width=.9\linewidth]{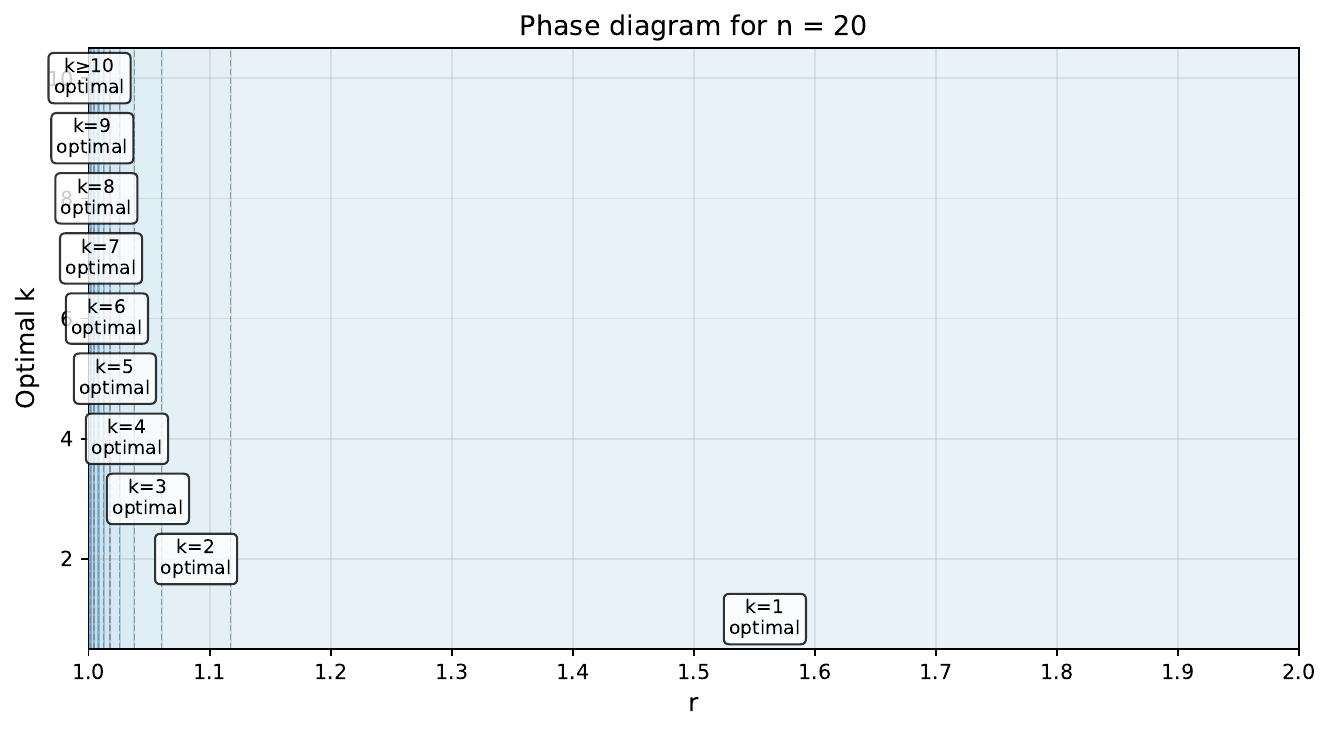}}
\caption{In each region $(r_k(n), r_{k-1}(n))$, the $k$-isolated cut achieves maximum weight among all isolated cuts.}
\label{fig:test}
\end{figure}

\subsection{Scaling Law}\label{sec:scaling}

The results above are rigorous; this subsection is heuristic and serves only to illuminate the asymptotic behavior of the thresholds.

The thresholds admit the following approximation.

By balancing the positive and 
negative terms in $P^{n,k}(r)$ through comparison of term counts and average 
exponents, one obtains

\begin{equation}\label{eq:scaling}
r_k(n) - 1 \approx \frac{\ln\mleft(\frac{n-k-1}{k}\mright)}{\Delta(n,k)} \qquad \text{as } n \to \infty,
\end{equation}
where $\Delta(n,k) = \frac{(k+2)(n-k)}{2} - \frac{(k+1)(4-k)}{6}$. 
For fixed $k$, this predicts decay of order $\frac{\ln n}{n}$, 
slower than any power law. The derivation and numerical comparisons 
appear in Appendix~\ref{app:scaling}.


\vskip5mm
The results above establish that isolated cuts dominate among themselves. The natural next question is whether they also dominate among \emph{all} cuts.

\section{Global optimality: conjecture and computational evidence}\label{sec:global}

We conjecture that the phase diagram established in Theorem~\ref{thm:isolated_optimality} 
extends to all cuts: the $k$-isolated cut that wins among isolated cuts also wins globally.

\begin{conjecture}[Global Optimality]\label{conj:global}
For $n \geq 7$ and $r \in (r_k(n), r_{k-1}(n))$, the $k$-isolated cut $C_k$ achieves the maximum weight among all $2^{n-1}$ cuts.
\end{conjecture}

The conjecture has been verified by exhaustive enumeration of all $2^{n-1}$ cuts for $n \leq 21$ and by near-isolated verification for $n \leq 100$. 
For exhaustive enumeration, we used a uniform grid over $(1.001, 1.999)$ with mesh size $0.001$ (1998 evaluation points); for each $(n,r)$ pair, we confirmed that the maximum weight is achieved by an isolated cut when $n \geq 7$.

The bound $n \geq 7$ is sharp. For $n \in \{4, 5, 6\}$, exhaustive search reveals intervals where 
non-isolated cuts are globally optimal. In each case, these are \emph{near-isolated cuts} of the 
form $S_k^* = \{1, \ldots, k, n\}$---partitions that include vertex $n$ instead of a contiguous block.
Table~\ref{tab:small_n} summarizes the results. 


\begin{table}[!ht]
\centering
\caption{Globally optimal cuts for $n = 4, 5, 6$ by exhaustive enumeration. 
At interval boundaries, the adjacent cuts have equal weight.}
\label{tab:small_n}
\medskip
\begin{tabular}{@{}clll@{}}
\toprule
$n$ & Interval & Optimal cut & Type \\
\midrule
4 & $(1, 1.324)$ & $S_1^* = \{1, 4\}$ & near-isolated \\
  & $(1.324, 2)$ & $C_1 = \{1\}$ & isolated \\
\addlinespace
5 & $(1, 1.22)$ & $C_2 = \{1, 2\}$ & isolated \\
  & $(1.22, 1.28)$ & $S_1^* = \{1, 5\}$ & near-isolated \\
  & $(1.28, 2)$ & $C_1 = \{1\}$ & isolated \\
\addlinespace
6 & $(1, 1.04)$ & $C_3 = \{1, 2, 3\}$ & isolated \\
  & $(1.04, 1.06)$ & $S_2^* = \{1, 2, 6\}$ & near-isolated \\
  & $(1.06, 1.24)$ & $C_2 = \{1, 2\}$ & isolated \\
  & $(1.24, 2)$ & $C_1 = \{1\}$ & isolated \\
\bottomrule
\end{tabular}
\end{table}

Figure~\ref{fig:small_n} visualizes these results. For $n = 4$, the isolated cut $C_2$ is 
\emph{never} globally optimal---it is dominated by $S_1^*$ throughout its theoretical region 
of optimality. For $n = 5$ and $n = 6$, near-isolated cuts intrude in intervals surrounding the 
thresholds $r_k(n)$ where consecutive isolated cuts exchange optimality. The intrusion intervals 
shrink as $n$ increases, and disappear entirely for $n \geq 7$.

\begin{figure}[!ht]
\centering
\begin{tikzpicture}[xscale=7, yscale=1.1]

\definecolor{isol}{RGB}{70, 130, 180}
\definecolor{near}{RGB}{205, 92, 92}
\def\h{0.18}

\def\y{3.3}
\fill[near!40] (1,\y-\h) rectangle (1.324,\y+\h);
\fill[isol!40] (1.324,\y-\h) rectangle (2,\y+\h);
\draw[thick] (1,\y) -- (2,\y);
\node[left, font=\small] at (0.98,\y) {$n=4$};
\node[above, font=\scriptsize] at (1.16,\y+\h) {$S_1^*$};
\node[above, font=\scriptsize] at (1.66,\y+\h) {$C_1$};
\draw (1.324,\y-\h) -- (1.324,\y+\h);
\node[below, font=\tiny] at (1.324,\y-\h-0.02) {$1.32$};

\def\y{2.2}
\fill[isol!40] (1,\y-\h) rectangle (1.22,\y+\h);
\fill[near!40] (1.22,\y-\h) rectangle (1.278,\y+\h);
\fill[isol!40] (1.278,\y-\h) rectangle (2,\y+\h);
\draw[thick] (1,\y) -- (2,\y);
\node[left, font=\small] at (0.98,\y) {$n=5$};
\node[above, font=\scriptsize] at (1.11,\y+\h) {$C_2$};
\node[above, font=\scriptsize] at (1.249,\y+\h) {$S_1^*$};
\node[above, font=\scriptsize] at (1.64,\y+\h) {$C_1$};
\draw (1.22,\y-\h) -- (1.22,\y+\h);
\draw (1.278,\y-\h) -- (1.278,\y+\h);
\node[below, font=\tiny] at (1.2,\y-\h-0.02) {$1.22$};
\node[below, font=\tiny] at (1.3,\y-\h-0.02) {$1.28$};

\def\y{1.1}
\fill[isol!40] (1,\y-\h) rectangle (1.041,\y+\h);
\fill[near!40] (1.041,\y-\h) rectangle (1.062,\y+\h);
\fill[isol!40] (1.062,\y-\h) rectangle (2,\y+\h);
\draw[thick] (1,\y) -- (2,\y);
\node[left, font=\small] at (0.98,\y) {$n=6$};
\node[above, font=\scriptsize] at (1,\y+\h) {$C_3$};
\node[above, font=\scriptsize] at (1.06,\y+\h+0.06) {$S_2^*$};
\node[above, font=\scriptsize] at (1.15,\y+\h) {$C_2$};
\node[above, font=\scriptsize] at (1.62,\y+\h) {$C_1$};
\draw (1.041,\y-\h) -- (1.041,\y+\h);
\draw (1.062,\y-\h) -- (1.062,\y+\h);
\draw (1.243,\y-\h) -- (1.243,\y+\h);
\node[below, font=\tiny] at (1.021,\y-\h-0.02) {$1.04$};
\node[below, font=\tiny] at (1.09,\y-\h-0.02) {$1.06$};
\node[below, font=\tiny] at (1.243,\y-\h-0.02) {$1.24$};

\def\y{0}
\fill[isol!40] (1,\y-\h) rectangle (2,\y+\h);
\draw[thick] (1,\y) -- (2,\y);
\node[left, font=\small] at (0.98,\y) {$n \geq 7$};
\node[above, font=\scriptsize] at (1.06,\y+\h) {$C_{\lfloor n/2\rfloor}$};
\node[above, font=\scriptsize] at (1.22,\y+\h) {$\cdots$};
\node[above, font=\scriptsize] at (1.45,\y+\h) {$C_2$};
\node[above, font=\scriptsize] at (1.75,\y+\h) {$C_1$};

\node[below, font=\scriptsize] at (1,-0.32) {$1$};
\node[below, font=\scriptsize] at (2,-0.32) {$2$};
\node[below, font=\scriptsize] at (1.5,-0.32) {$r$};
\draw[->] (1,-0.34) -- (2.05,-0.34);

\fill[isol!40] (1.35,-0.65) rectangle (1.45,-0.85);
\node[right, font=\scriptsize] at (1.45,-0.8) {Isolated $C_k$};
\fill[near!40] (1.72,-0.65) rectangle (1.82,-0.85);
\node[right, font=\scriptsize] at (1.82,-0.8) {Near-isolated $S_k^*$};

\end{tikzpicture}
\caption{Globally optimal cuts as a function of $r$, determined by exhaustive enumeration. For $n \leq 6$, near-isolated cuts $S_k^* = \{1,\ldots,k,n\}$ are optimal in certain intervals (red regions). For $n \geq 7$, isolated cuts are globally optimal throughout $(1,2)$ (Conjecture~\ref{conj:global}).}
\label{fig:small_n}
\end{figure}
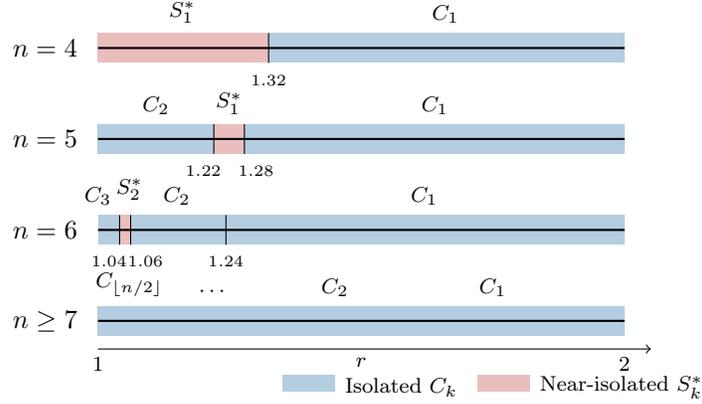


For $n = 4$, the comparison between $C_2 = \{1,2\}$ and $S_1^* = \{1,4\}$ yields a simple identity. 
Direct calculation gives
\[
W^4(S_1^*; r) - W^4(C_2; r) = r^5 - r^3 - r^2 + 1 = (r-1)^2(r+1)(r^2+r+1),
\]
which is strictly positive for all $r > 1$. Thus $C_2$ is \emph{never} optimal---it loses to 
$S_1^*$ throughout $(1,2)$. For $n \geq 5$, the analogous difference $W^n(S_1^*;r) - W^n(C_2;r)$ no 
longer has $(r-1)^2$ as a factor, so $C_2$ can beat $S_1^*$ in part of the interval. The region 
where isolated cuts win expands with $n$, and by $n = 7$ it covers all of $(1,2)$.


\subsection{Computational Verification}\label{sec:computational}

We employed two verification strategies to support Conjecture~\ref{conj:global}:
exhaustive enumeration for $n \leq 21$ and targeted near-isolated verification 
for $n \leq 100$. No violations were found for any $n \geq 7$. Full threshold 
tables, methodology details, and numerical considerations are provided in 
Appendix~\ref{app:computational}. 

\section{Conclusion}\label{sec:conclusion}

We have studied the maximum cut problem on $K_n$ with geometric edge weights $w_i = r^{N-i}$. 
For $r \in (1,2)$, we proved that isolated cuts $C_k$ are optimal among themselves, with 
transitions at thresholds $r_1(n) > r_2(n) > \cdots > 1$. We conjecture this extends to 
global optimality for $n \geq 7$, supported by computation up to $n = 100$. The bound is 
sharp: for $n \leq 6$, near-isolated cuts $S_k^* = \{1,\ldots,k,n\}$ can win.

Several questions remain. Can one prove global optimality, perhaps by first showing that 
near-isolated cuts are the only competitors? Can the scaling law~\eqref{eq:scaling} be 
proved with explicit error bounds? Does similar threshold structure appear for other 
weight families?

To place our results in context, we compare the optimal cuts against generic lower bounds for 
weighted Max-Cut. Let $\mu(G)$ denote the maximum cut value, $w(G)$ the total edge weight, and $M$ a maximum-weight matching (a set of pairwise non-adjacent edges of maximum total weight). The strongest known bounds include: Poljak--Turz\'{\i}k~\cite{poljak1986polynomial}, 
$\mu(G) \geq \frac{w(G)}{2} + \frac{w(T_{\min})}{4}$ where $T_{\min}$ is a minimum spanning tree; and Gutin--Yeo~\cite{gutin2023lower}, 
$\mu(G) \geq \frac{w(G) + w(M)}{2}$. 
For our geometric weights, the Gutin--Yeo bound is tightest; Table~\ref{tab:bounds} compares it with the optimal isolated cuts for several parameter values.

\begin{table}[!ht]
\centering
\caption{Gutin--Yeo bound vs.\ optimal $k$-isolated cut.}
\label{tab:bounds}
\begin{tabular}{cccccc}
\toprule
$n$ & $r$ & $k$ & Optimum & GY bound & Gap \\
\midrule
8 & 1.05 & 3 & 35.88 & 33.20 & 7.5\% \\
8 & 1.10 & 2 & 89.33 & 76.86 & 14.0\% \\
8 & 1.50 & 1 & $1.60 \times 10^5$ & $1.14 \times 10^5$ & 29.1\% \\
10 & 1.20 & 1 & $1.47 \times 10^4$ & $1.07 \times 10^4$ & 27.1\% \\
\bottomrule
\end{tabular}
\end{table}

Even the strongest known generic bound lies 7--31\% below the true optimum. This gap illustrates 
why bounds valid for all weighted graphs cannot capture the fine structure present in specific 
weight families---and why exact characterizations, as developed in this paper, are valuable. 
More broadly, this work demonstrates that structured weight functions can induce rich 
combinatorial behavior even on complete graphs. 


\appendix

\section{Scaling Law Derivation}\label{app:scaling}

We derive the heuristic approximation~\eqref{eq:scaling} for threshold values.
Recall that $P^{n,k}(r) = W^n(C_k; r) - W^n(C_{k+1}; r)$ can be written as
\[
P^{n,k}(r) = \sum_{i=1}^{k} r^{a_i} - \sum_{j=1}^{n-k-1} r^{b_j},
\]
where the positive terms consist of $k$ monomials with exponents $a_1 > \cdots > a_k$, 
the negative terms consist of $n-k-1$ monomials with exponents $b_1 > \cdots > b_{n-k-1}$, 
and all positive exponents exceed all negative exponents ($a_k > b_1$). 

\begin{example}
For $n = 7$ and $k = 2$: $P^{7,2}(r) = r^{19} + r^{14} - (r^9 + r^8 + r^7 + r^6)$.
\end{example}

The explicit formulas for the exponents follow from~\eqref{eq:edge_index}:
\begin{align*}
a_i &= \tbinom{n-i+1}{2} - (k-i+1), \quad i = 1, \ldots, k, \\
\{b_j\} &= \mleft\{\tbinom{m}{2} - 2m + 3, \, \tbinom{m}{2} - 2m + 4, \, \ldots, \, \tbinom{m}{2} - m\mright\},
\end{align*}
where $m = n - k + 1$. The negative exponents form an arithmetic sequence of length $n - k - 1$.

\subsection*{Averaging heuristic}

At the threshold $r = r_k(n)$, we have $\sum_{i=1}^{k} r^{a_i} = \sum_{j=1}^{n-k-1} r^{b_j}$.
The left side has few terms with high exponents, while the right side has many terms with 
low exponents. Approximating each sum by the number of terms times $r$ raised to the 
average exponent:
\[
k \cdot r^{\bar{a}} \approx (n-k-1) \cdot r^{\bar{b}},
\]
where $\bar{a} = \frac{1}{k}\sum_{i=1}^{k} a_i$ and $\bar{b} = \frac{1}{n-k-1}\sum_{j=1}^{n-k-1} b_j$. 
Writing $\Delta = \bar{a} - \bar{b}$ for the gap between average exponents, we obtain 
$k \cdot r^{\Delta} \approx n - k - 1$. Taking logarithms and using $\ln(r) \approx r - 1$ 
for $r$ close to 1, we obtain
\begin{equation}\label{eq:scaling_app}
r_k(n) - 1 \approx \frac{\ln\mleft(\frac{n-k-1}{k}\mright)}{\Delta(n,k)}.
\end{equation}

Since the negative exponents form an arithmetic sequence from $\binom{m}{2} - 2m + 3$ 
to $\binom{m}{2} - m$ (where $m = n - k + 1$), their average is 
$\bar{b} = \frac{(n-k)(n-k-2)}{2}$. The average of positive exponents is 
$\bar{a} = \binom{n}{2} - \frac{(k-1)n}{2} - \frac{(k+1)(4-k)}{6k}$. Combining:

\begin{proposition}\label{prop:delta}
The gap between average exponents is given by
\[
\Delta(n,k) = \frac{(k+2)(n-k)}{2} - \frac{(k+1)(4-k)}{6}.
\]
\end{proposition}

\begin{example}
For $n = 7$, $k = 2$: $\Delta(7,2) = \frac{4 \cdot 5}{2} - \frac{3 \cdot 2}{6} = 10 - 1 = 9$.
Direct calculation gives $\bar{a} = \frac{19 + 14}{2} = 16.5$ and $\bar{b} = \frac{9+8+7+6}{4} = 7.5$, 
so $\Delta = 9$ as predicted.
\end{example}

For large $n$ with fixed $k$, the dominant term is $\Delta(n,k) \approx \frac{(k+2)n}{2}$, giving
\[
r_k(n) - 1 \approx \frac{2\ln(n/k)}{(k+2)n} \qquad \text{as } n \to \infty.
\]

\subsection*{Comparison with computed values}

Table~\ref{tab:scaling} compares actual threshold values with predictions from~\eqref{eq:scaling_app}. 
The formula achieves 1--8\% accuracy for $k \geq 2$, with a systematic $\sim$10\% underestimate 
for $k = 1$. This error reflects the fact that a single positive exponent cannot be meaningfully \enquote{averaged.}

\begin{table}[!ht]
\centering
\begin{tabular}{cc|ccc}
\toprule
$n$ & $k$ & $r_k(n) - 1$ (actual) & Prediction & Error \\
\midrule
10 & 1 & 0.1883 & 0.1664 & $-11.6\%$ \\
10 & 2 & 0.0843 & 0.0835 & $-1.0\%$ \\
10 & 3 & 0.0404 & 0.0412 & $+2.0\%$ \\
\midrule
15 & 1 & 0.1439 & 0.1282 & $-10.9\%$ \\
15 & 2 & 0.0713 & 0.0717 & $+0.5\%$ \\
15 & 3 & 0.0424 & 0.0443 & $+4.4\%$ \\
\midrule
20 & 1 & 0.1173 & 0.1051 & $-10.4\%$ \\
20 & 2 & 0.0603 & 0.0611 & $+1.5\%$ \\
20 & 3 & 0.0378 & 0.0400 & $+5.9\%$ \\
\midrule
30 & 2 & 0.0461 & 0.0473 & $+2.7\%$ \\
30 & 3 & 0.0300 & 0.0323 & $+7.8\%$ \\
30 & 4 & 0.0214 & 0.0235 & $+10.0\%$ \\
\bottomrule
\end{tabular}
\caption{Comparison of actual threshold values with the scaling law prediction~\eqref{eq:scaling_app}.}
\label{tab:scaling}
\end{table}

\begin{figure}[!ht]
\centering
\includegraphics[width=0.75\textwidth]{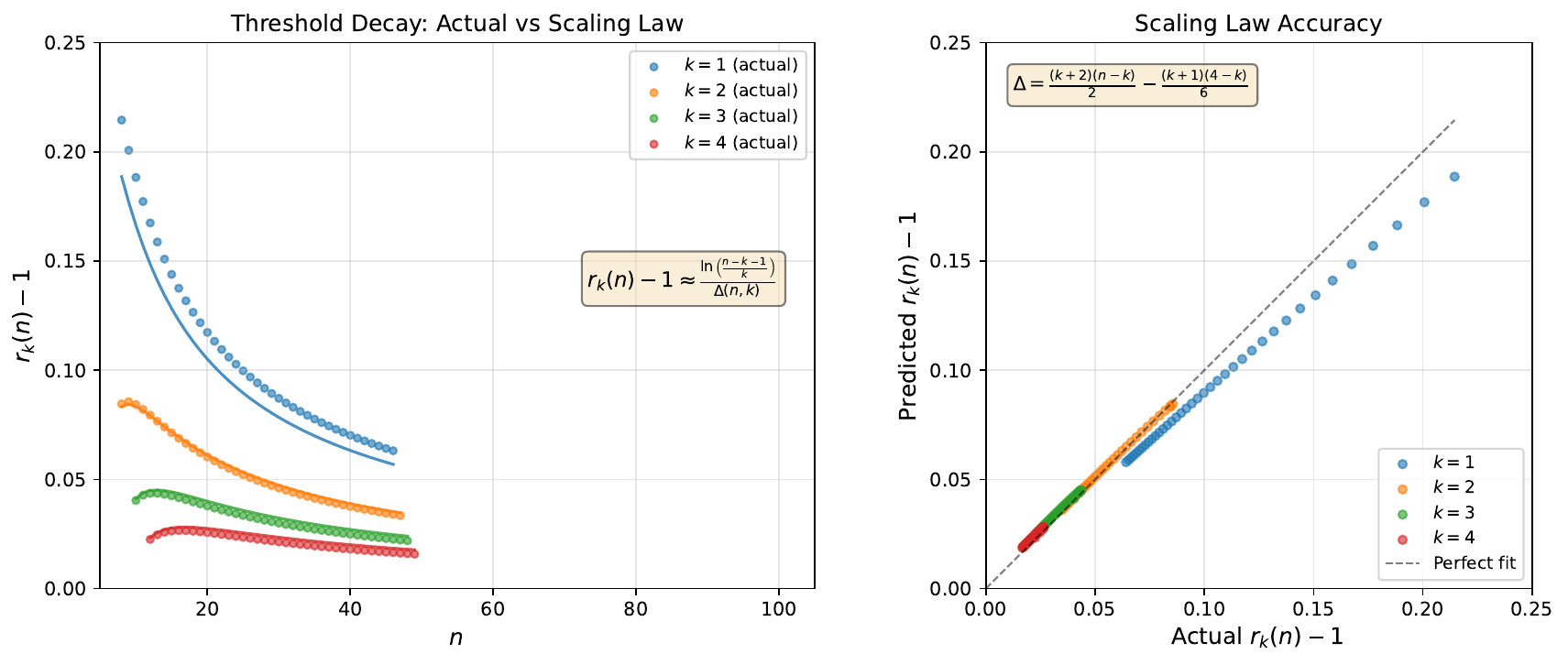}
\caption{Verification of the scaling law. \emph{Left:} The quantity $r_k(n) - 1$ 
as a function of $n$ for $k = 1, 2, 3, 4$. Points show computed values; curves 
show the prediction~\eqref{eq:scaling_app}. \emph{Right:} Parity plot comparing 
predicted vs.\ actual values.}
\label{fig:scaling}
\end{figure}


\section{Computational Details}\label{app:computational}

\subsection*{Threshold tables}

Table~\ref{tab:thresholds} provides computed threshold values $r_k(n)$ for $n = 6$ to $20$ 
and all valid $k$. Values are computed to six decimal places using the \texttt{uniroot} 
function in R with tolerance $10^{-12}$.

\begin{table}[!ht]
\centering
\caption{Threshold values $r_k(n)$ for $6 \leq n \leq 20$. Entry \enquote{--} indicates $k > \lfloor n/2 \rfloor - 1$.}
\label{tab:thresholds}
\small
\begin{tabular}{c|cccccccc}
\hline
$n$ & $k=1$ & $k=2$ & $k=3$ & $k=4$ & $k=5$ & $k=6$ & $k=7$ & $k=8$ \\
\hline
6 & 1.243347 & 1.058812 & -- & -- & -- & -- & -- & -- \\
7 & 1.229318 & 1.078366 & -- & -- & -- & -- & -- & -- \\
8 & 1.214506 & 1.084615 & 1.024141 & -- & -- & -- & -- & -- \\
9 & 1.200695 & 1.085595 & 1.035149 & -- & -- & -- & -- & -- \\
10 & 1.188280 & 1.084320 & 1.040361 & 1.012226 & -- & -- & -- & -- \\
11 & 1.177240 & 1.082068 & 1.042719 & 1.018839 & -- & -- & -- & -- \\
12 & 1.167434 & 1.079426 & 1.043572 & 1.022554 & 1.007037 & -- & -- & -- \\
13 & 1.158700 & 1.076677 & 1.043595 & 1.024651 & 1.011279 & -- & -- & -- \\
14 & 1.150889 & 1.073959 & 1.043150 & 1.025794 & 1.013920 & 1.004417 & -- & -- \\
15 & 1.143869 & 1.071340 & 1.042441 & 1.026351 & 1.015585 & 1.007290 & -- & -- \\
16 & 1.137530 & 1.068850 & 1.041587 & 1.026538 & 1.016632 & 1.009206 & 1.002953 & -- \\
17 & 1.131778 & 1.066499 & 1.040660 & 1.026488 & 1.017272 & 1.010504 & 1.004984 & -- \\
18 & 1.126535 & 1.064288 & 1.039702 & 1.026284 & 1.017638 & 1.011386 & 1.006409 & 1.002071 \\
19 & 1.121736 & 1.062212 & 1.038741 & 1.025980 & 1.017816 & 1.011983 & 1.007424 & 1.003558 \\
20 & 1.117326 & 1.060264 & 1.037794 & 1.025612 & 1.017863 & 1.012377 & 1.008152 & 1.004643 \\
\hline
\end{tabular}
\end{table}

Several patterns are evident. For fixed $n$, thresholds decrease with $k$: 
$r_1(n) > r_2(n) > \cdots$ (Theorem~\ref{thm:monotonicity}). For fixed $k=1$, 
thresholds decrease monotonically with $n$, but for $k \geq 2$, thresholds are 
not monotone in $n$; for instance, $r_2(9) > r_2(8)$. All thresholds converge 
to $1$ as $n \to \infty$ (Theorem~\ref{thm:convergence}).

Table~\ref{tab:n50} shows all threshold values for $n = 50$.

\begin{table}[!ht]
\centering
\caption{All threshold values $r_k(n)$ for $n=50$.}
\label{tab:n50}
\begin{tabular}{c|c||c|c||c|c||c|c}
\toprule
$k$ & $r_k(n)$ & $k$ & $r_k(n)$ & $k$ & $r_k(n)$ & $k$ & $r_k(n)$ \\
\midrule
1 & 1.059214 & 7 & 1.007883 & 13 & 1.003032 & 19 & 1.001136 \\
2 & 1.031923 & 8 & 1.006577 & 14 & 1.002619 & 20 & 1.000908 \\
3 & 1.021270 & 9 & 1.005558 & 15 & 1.002257 & 21 & 1.000694 \\
4 & 1.015563 & 10 & 1.004741 & 16 & 1.001934 & 22 & 1.000488 \\
5 & 1.012020 & 11 & 1.004071 & 17 & 1.001644 & 23 & 1.000290 \\
6 & 1.009616 & 12 & 1.003510 & 18 & 1.001380 & 24 & 1.000096 \\
\bottomrule
\end{tabular}
\end{table}

\subsection*{Verification methodology}

\paragraph{Exhaustive enumeration for small $n$.}
For $n \leq 21$, we performed exhaustive enumeration of all $2^{n-1}$ distinct cuts 
and computed their weights on a uniform grid over $(1.001, 1.999)$ with mesh size 
$\Delta r = 0.001$ (1998 evaluation points). For each $(n, r)$ pair, we verified 
that the maximum weight is achieved by a $k$-isolated cut. The conjecture holds 
for $n \geq 7$; the counterexamples for $n \in \{4, 5, 6\}$ are characterized 
completely in Section~\ref{sec:global}. Note that the uniform mesh becomes coarser 
than the spacing between consecutive thresholds for large $k$ (e.g., 
$r_{23}(50) - r_{24}(50) \approx 2 \times 10^{-4}$); in these regions the 
near-isolated verification provides the primary evidence.

\paragraph{Near-isolated verification for larger $n$.}
For $n > 21$, exhaustive enumeration becomes infeasible. However, the analysis 
in Section~\ref{sec:global} identifies near-isolated cuts $S^*_k = \{1, \ldots, k, n\}$ 
as the only non-isolated cuts that achieve global optimality for $n \leq 6$. 
We therefore performed targeted verification: for each $n$ from 7 to 100, we 
computed all thresholds $r_k(n)$ and verified that every near-isolated cut 
$S^*_j$ ($j = 1, \ldots, n-2$) is beaten by the optimal isolated cut at 
20 uniformly spaced points within each threshold interval $(r_k(n), r_{k-1}(n))$.

\paragraph{Numerical considerations.}
Two technical issues required careful handling:
\begin{itemize}
\item Thresholds $r_k(n)$ for large $k$ can be extremely close to 1; for instance, 
$r_{24}(50) - 1 \approx 10^{-4}$. The root-finding algorithm must therefore 
search in the interval $(1 + 10^{-9}, 2)$ rather than starting at $r = 1.001$.

\item For large $n$, direct weight computation causes numerical overflow 
(e.g., $r^{N-1}$ with $N = \binom{100}{2} = 4950$). Let $W = \sum_i r^{e_i}$ be the weight of a cut, where $e_i$ are the exponents 
of its crossing edges. Weights are computed as 
\[
\log W = \max_i(e_i \log r) + \log\sum_i \exp((e_i - \max_j e_j)\log r),
\]
and comparisons are performed in log-space.
\end{itemize}

\paragraph{Results.}
No violations were found for any $n$ from 7 to 100. Combined with the 
exhaustive verification for $n \leq 21$, this provides strong computational 
evidence for Conjecture~\ref{conj:global}.

\paragraph{Code availability.}
All verification code is available at:\\ \url{https://github.com/nevmaric/lexicut}.

\section*{Acknowledgements}
The author is thankful to Tatjana Davidovi\'c and Dragan Uro\v{s}evi\'c for insightful discussions.

\printbibliography

\end{document}